\documentclass{birkjour}

\usepackage{amssymb, amsmath, amsfonts, epsfig, color,amscd,amsthm,setspace,lineno,labelfig}

\newtheorem{theorem}{Theorem}[section]

\newtheorem{lemma}[theorem]{Lemma}

\theoremstyle{definition}

\begin{document}

\title{Gaps in the space of skeletal signatures}

\author{James W Anderson}
\address{Mathematical Sciences\\University of Southampton\\Southampton SO17 1BJ England}
\email{j.w.anderson@southampton.ac.uk}

\author{Aaron Wootton}
\address{Department of Mathematics\\University of Portland\\Portland OR 97203 USA}
\email{wootton@up.edu}

\thanks{This work was supported in part by EPSRC grant EP/K033239/1.  Both authors would like to thanks the other's institution for its hospitality during the work that led to this paper.}


\begin{abstract} Skeletal signatures were introduced in \cite{AndWoo} as a tool to describe the space of all signatures with which a group can act on a surface of genus $\sigma \geq 2$. In the present paper we provide a complete description of the gaps that appear in the space of skeletal signatures, together with proofs of the conjectures posed in \cite{AndWoo}. 
\end{abstract}

\maketitle

\section{Introduction and preliminaries}
\label{introduction}

In \cite{AndWoo}, the authors introduced the notion of a {\em skeletal signature} for the action of a finite group on a closed Riemann surface of genus $\sigma\ge 2$, together with the resultant space $\mathcal{K}_\sigma$ of skeletal signatures, as tools for understanding the space of all signatures of all actions of finite groups on closed Riemann surfaces of a fixed genus $\sigma\ge 2$.   The study of automorphism groups of Riemann surfaces and the signatures of such actions has been the focus of much study since the early days of the development of Riemann surfaces, and remains a vibrant area of activity.  We do note that most of the current activity focuses on properties of particular groups or families of groups, rather than on the genus of the surface being acted upon.

To set  notation, let $X$ be a closed Riemann surface of genus $\sigma\ge 2$, and let $G$ be a (necessarily finite) group acting on $X$ by conformal homeomorphisms.   Let $(h; n_1,\ldots, n_r)$ be the {\em signature} of the action of $G$ on $X$, by which we mean that the quotient surface orbifold $X/G$ has genus $h$ and the orbifold covering $X\rightarrow X/G$ is branched over $r$ points with orders $n_1,\ldots, n_r$.  

To the signature $(h; n_1,\ldots, n_r)$, we associate the {\em skeletal signature} $(h,r)$, which is the ordered pair consisting of the genus of the quotient surface orbifold and the number of branch points.  Define $\mathcal{K}_\sigma$ to be the set of all skeletal signatures associated to all groups acting on all closed Riemann surfaces of a fixed genus $\sigma$. 

The main reason for the interest in these skeletal signatures is that they provide a reasonable summary of the information contained in the corresponding full signatures, and skeletal signatures are significantly more tractable than full signatures.  One demonstration of this is contained below, in our description of the well-formed gaps that appear in the space of skeletal signatures for all genera.  To our knowledge, no corresponding result is known for full signatures.  

The main significance of these gaps is that if the point $(h_0,r_0)$ cannot be realized as a skeletal signature for genus $\sigma$, then we are able to exclude a large number of potential signatures.  Specifically, there cannot exist any group $G$ acting on any closed Riemann surface $X$ of genus $\sigma$ for which the action of $G$ on $X$ has signature $(h_0; n_1,\ldots, n_{r_0})$ for any $n_1,\ldots, n_{r_0}$. 

Referring the reader to Figure \ref{fig:example}, one curious and notable feature of $\mathcal{K}_\sigma$ is the existence of well-formed gaps in which no skeletal signatures appear.  While we present only $\mathcal{K}_{48}$ in this Figure, such gaps appear in all genera.  The data from which this Figure was constructed arises from the work of Breuer \cite{Breu} and his classification of all signatures arising from actions of groups on all closed Riemann surfaces of genera $2\le \sigma\le 48$.  This data is available via the {\bf genus} package in GAP \cite{GAP}.  Indeed, it was from our attempts to visualize Breuer's data that led us to skeletal signatures initially.

The main technical result used in the descriptions of the gaps and of the proofs of the Conjectures from \cite{AndWoo} is the Riemann-Hurwitz formula.  Remarkably, although satisfaction of the Riemann-Hurwitz formula is not sufficient to guarantee the existence of a group action yielding a given signature, current evidence suggests that satisfaction of the Riemann-Hurwitz formula is nearly always sufficient to guarantee the existence of a group action yielding a given skeletal signature.  We will explore this in detail in future work \cite{AndWooskeletal}.

While we focus primarily on the gaps that arise for all genera,  we will also discuss in Section \ref{sporadic r=1} some results for sporadic points, which are those points in the $(h,r)$-plane which correspond to skeletal signatures for infinitely many genera and which do not arise as the skeletal signatures for infinitely many genera as well.  In particular, we resolve Conjecture 3.15 from \cite{AndWoo} by showing that all skeletal signatures $(h,1)$ are sporadic for all $h\ge 2$.  

We close this Introduction by stating the Riemann-Hurwitz formula.  Let $X$ be a compact Riemann surface of genus $\sigma\ge 2$, let $G$ be a group acting on $X$ by conformal homeomorphisms, and suppose that the signature of the action of $G$ on $X$ is $(h; n_1,\ldots, n_r)$.  We then have that 
\begin{eqnarray}
\sigma -1 = |G| \left( h - 1+\frac{r}{2} - \frac{1}{2} \sum_{j=1}^r \frac{1}{n_j} \right).
\label{riemann hurwitz}
\end{eqnarray}

\section{Triangles of skeletal signatures in the $(h,r)$-plane}
\label{triangles possibilities}

Fix a genus $\sigma\ge 2$.  In this Section, we describe the gaps in the $(h,r)$-plane, showing that their existence is forced by the Riemann-Hurwitz formula.  By a {\em gap}, we mean a subset of the $(h,r)$-plane, determined in terms of $\sigma$, which contains no points of the space $\mathcal{K}_\sigma$ of skeletal signatures for genus $\sigma$.   

We begin with the following  construction.  Take an integer $N\ge 2$ and consider the Riemann-Hurwitz formula for the action of any group $G$ of order $N$ on a closed Riemann surface $X$ of genus $\sigma\ge 2$, where $G$ acts on $X$ with signature $(h; n_1,\ldots, n_r)$.   In particular, we do not pay particular attention to any aspect of $G$ other than its order $|G| =N$. 

The ramifications orders $n_j$ of the covering $X\rightarrow X/G$ trivially satisfy the inequalities $2\le n_j\le N = |G|$ for all $1\le j\le r$.
(Though we do not need to do so here, we note that it is straightforward to refine this estimate for a specific group $G$, for instance by bounding the ramification orders by the minimum and maximum orders of non-trivial elements of $G$.)  Substituting this into equation (\ref{riemann hurwitz}) and rearranging, we see for a fixed $h$ that
\begin{eqnarray}
\frac{2(\sigma -1 +N(1-h))}{N-1}\le r\le \frac{4 (\sigma -1 + N(1-h))}{N}. 
\label{two lines}
\end{eqnarray}
We interpret these inequalities as giving, for the fixed $\sigma\ge 2$ and for the given $h\ge 0$, the $r$-coordinates for all of the skeletal signatures $(h,r)$ that might arise from the action of some group of order $N$ on $X$.  That is, we are covering the space ${\mathcal K}_\sigma$ of skeletal signatures by these vertical slices.

Equation (\ref{two lines}) yields a pair of lines.  The left-hand inequality yields the line 
\[ L^\sigma_N = \{ 2\sigma -2 + 2N = (N-1)r +2N h\} \]
and the right-hand side yields the line 
\[ U^\sigma_N = \{ 4(N +\sigma-1) = N r + 4 Nh\}.\]
Calculating, we see that $L^\sigma_N$ and $U^\sigma_N$ intersect at the point $(1 +\frac{\sigma-1}{N}, 0)$.

Define $\mathcal{P}_\sigma (N)$ to be the closed triangular region in the $(h,r)$-plane bounded by the lines   $L^\sigma_N$ and $U^\sigma_N$; that is, we define $\mathcal{P}_\sigma (N)$ to be all those points $(h,r)$ for which 
\[ \frac{2(\sigma -1 +N(1-h))}{N-1}\le r\le \frac{4 (\sigma -1 + N(1-h))}{N} \]
for all $0\le h\le 1+\frac{\sigma-1}{N}$.  We refer to $L^\sigma_N$ as the {\em lower line} for $\mathcal{P}_\sigma (N)$ and to $U^\sigma_N$ as the {\em upper line} for $\mathcal{P}_\sigma (N)$, for the obvious reason that $\mathcal{P}_\sigma (N)$ lies above $L^\sigma_N$ and below $U^\sigma_N$. We note that 
\[ \mathcal{K}_\sigma \subset  \bigcup_{N\ge 2} \mathcal{P}_\sigma (N). \]

For $N=2$, observe that $L^\sigma_2$ and $U^\sigma_2$ coincide, as both are the line $\{ 2(\sigma -1) = r + 4h\}$.  This is not a suprise, given the method we used to generate the lines. In fact, this is a welcome observation, as we know from our previous work that all of the skeletal signatures arising from $C_2$-actions on surfaces of genus $\sigma$ lie on this line.

This observation holds more generally, for any group of the form $(C_p)^n$ for a prime $p\ge 2$ and an integer $n\ge 1$.  As an illustration of how we might refine the argument given above, we give the proof of this observation.

\begin{lemma}  For a prime $p\ge 2$ and an integer $n\ge 1$, suppose that $(C_p)^n$ acts by conformal homeomorphisms on the closed Riemann surface $X$ of genus $\sigma\ge 2$.  The skeletal signature $(h,r)$ of the action of $(C_p)^n$ on $X$ with signature $(h; n_1,\ldots, n_r)$ then lies on the line 
\[ L_\sigma (p,n) = \{  2p^n h +  (p-1) p^{n-1} r = 2p^n - 2 + 2\sigma\}. \]
In particular, the triangle $\mathcal{P}_\sigma ((C_p)^n)$ collapses to the line $L_\sigma (p,n)$. 
\label{single line}
\end{lemma}

\begin{proof}  The key to this observation is that all of the non-trivial elements of $(C_p)^n$ have order $p$, and so all of the ramification orders satisfy $n_j = p$.   Equation (\ref{riemann hurwitz}) then becomes
\[ \sigma - 1 = p^n \left(h -1 + \frac{r(p-1)}{2p} \right), \]
which is the line $$L_\sigma (p,n) = \{ 2p^n h +  (p-1) p^{n-1} r = 2p^n - 2 + 2\sigma\}. \eqno\qedhere$$  
\end{proof}

To understand how gaps arise in the $(h,r)$-plane for a fixed genus $\sigma\ge 2$, we need to understand how the $\mathcal{P}_\sigma (N)$ behave relative to one another as $N$ varies.  We start by noting that the upper lines $U^\sigma_N$ are all parallel to one another over all values of $N$, as all have slope $-4$.  

Also, for a fixed $\sigma\ge 2$, the lower lines $L^\sigma_N$ are almost parallel as $N$ varies, as the slope of $L^\sigma_N$ is $\frac{N-1}{2N} = \frac{1}{2} -\frac{1}{2N}$.  Interestingly, the lower lines $L^\sigma_N$ all pass through the point $(h,r) = (\sigma, 2-2\sigma) = (2, \chi(X))$; since $\sigma\ge 2$, this immediately yields that the $L^\sigma_N$ are disjoint in the $(h,r)$-plane as $N$ varies.  However, we have not yet seen any particular consequence of the fact that the $L^\sigma_N$ all pass through the same point. 

Moreover, as $N$ increases, we have that the $U^\sigma_N$ and the $L^\sigma_N$ move to the left.  Combining these facts, we have the following observation regarding the behaviours of $\mathcal{P}_\sigma (M)$ and $\mathcal{P}_\sigma (N)$ for $M< N$. 

\begin{lemma} Fix a genus $\sigma\ge 2$ and let $3\le M < N$ be  integers. 
\begin{enumerate}
\item $\mathcal{P}_\sigma (M)$ lies strictly above the lower line $L^\sigma_N$ of $\mathcal{P}_\sigma (N)$.  
\item $\mathcal{P}_\sigma (N)$ lies strictly below the upper line $U^\sigma_M$ of $\mathcal{P}_\sigma (M)$.  
\end{enumerate}
\label{overlap of triangles}
\end{lemma}

\begin{proof} The proof of both parts of this Lemma follows immediately from the analysis that led to the definition of $\mathcal{P}_\sigma (N)$.  
\begin{enumerate}
\item The lower lines of $\mathcal{P}_\sigma (M)$ and $\mathcal{P}_\sigma(N)$, while not disjoint, intersect at the point $(h,r) = (\sigma, 2-2\sigma)$, which lies below the $h$-axis in the $(h,r)$-plane.  This forces $\mathcal{P}_\sigma (M)$ to lie in the complement of the lower line of $\mathcal{P}_\sigma (N)$.  The statement then follows from the observation that the lower line for $\mathcal{P}_\sigma (M)$ lies above the lower line for $\mathcal{P}_\sigma (N)$.
\item The upper lines of $\mathcal{P}_\sigma (M)$ and $\mathcal{P}_\sigma (N)$ are parallel, as both have slope $-4$, and hence are disjoint.   This forces $\mathcal{P}_\sigma (N)$ to lie in the complement of the upper line of $\mathcal{P}_\sigma (M)$.   The statement then follows from the observation that the upper line for $\mathcal{P}_\sigma (N)$ lies below the upper line for $\mathcal{P}_\sigma (M)$.
\end{enumerate}
\end{proof}

As a consequence of this, we see that the $\mathcal{P}_\sigma (N)$ overlap in a saw-tooth pattern, see Figure \ref{fig:example}.  With this in mind, define $\mathcal{G}_\sigma (N, N+1)$ to be the open triangular region in the $(h,r)$-plane bounded by the lines $U^\sigma_{N+1}$ and $L^\sigma_N$, lying to the right of their point $Q^\sigma_{N,N+1}$ of intersection, which we calculate to be 
\[ Q^\sigma_{N,N+1} = \left(\frac{(N-1)^2 +\sigma(N-3)}{(N-2)(N+1)}, \frac{4(\sigma -1)}{(N-2)(N+1)}\right). \]
Similarly, define $\mathcal{G}_\sigma (N,N+2)$ to be the open triangular region in the $(h,r)$-plane bounded by the lines $U^\sigma_{N+2}$ and $L^\sigma_N$, lying to the right of their point $Q^\sigma_{N,N+2}$ of intersection, which we calculate to be 
 \[ Q^\sigma_{N,N+2} = \left(\frac{N^2 -N +\sigma(N-4)}{N^2-4}, \frac{8(\sigma -1)}{N^2 -4}\right). \]

\begin{theorem} Fix a genus $\sigma \ge 2$ and let $N\ge 3$ be an integer.  In the case that $N+1$ is not prime, the region $\mathcal{G}_\sigma (N,N+1)$ is a gap in the $(h,r)$-plane, so that the intersection $\mathcal{G}_\sigma (N,N+1)\cap \mathcal{K}_\sigma$ is empty.

In the case that $N+1$ is prime, the region $\mathcal{G}_\sigma (N,N+2)$ is a gap in the $(h,r)$-plane,  with the exception of the points within this region lying on the line $L_\sigma (N+1,1)$ corresponding to the cyclic group of order $N+1$.  
\label{main gap theorem}
\end{theorem}

\begin{proof} The proof follows immediately by considering the relationship between the three triangular regions $\mathcal{P}_\sigma (N)$, $\mathcal{P}_\sigma (N+1)$ and $\mathcal{P}_\sigma (N+2)$.   For both  statements, the equations of the point of intersection follows immediately from the equations of the relevant lines, while the existence of the gaps follows from Lemma \ref{overlap of triangles}.   

For the second statement, the only additional point is that the middle triangle $\mathcal{P}_\sigma (N+1)$ collapses to the line $L_\sigma (N+1,1)$, which passes through the region $\mathcal{G}_\sigma (N,N+2)$ and so determines points of $\mathcal{K}_\sigma$ that lie in $\mathcal{G}_\sigma (N,N+2)$.
\end{proof}

The proof of Theorem \ref{main gap theorem} makes use of the region bounded by the lower line of $\mathcal{P}_\sigma (N)$ and the upper line of $\mathcal{P}_\sigma (N+1)$, which intersect at the point $Q^\sigma_{N,N+1}$, and so determine a pair of regions.  The gap consists of those points strictly between the two lines and to the right of the point of intersection, while the points between the two lines and to the left of the point of intersection are  skeletal signatures $(h,r)$ that correspond signatures for both order $N$ and order $N+1$ actions on some closed Riemann surface of genus $\sigma$. 

We close this Section by noting that when $N$ divides $\sigma -1$, the point $\left(  \frac{\sigma-1}{N} +1, 0\right)$ at which $\mathcal{P}_\sigma (N)$ meets $\{ r=0\}$ is always the skeletal signature of some group action of order $N$ on a surface of genus $\sigma$ (indeed, a cyclic group action, see \cite{Har1}). In fact, the point $\left(  \frac{\sigma-1}{N} +1, 0\right)$ is a skeletal signature for {\bf all } groups of order $N$ provided $\left(  \frac{\sigma-1}{N} +1\right)\geq n+1$ where $n$ is the largest power of any prime dividing $N$.  This follows using  an argument similar to Hurwitz's original argument that there exists a group action of every finite group on some surface, see \cite[Corollary 3.15]{Breu} and from the fact that the cardinality of a generating set for such a group is bounded by $n+1$, see \cite{Gur}.

\section{Gaps in the $(h,r)$-plane}
\label{gap conjectures}

The purpose of this Section is to consider the parts of the Conjectures from \cite{AndWoo} related to the existence of regular gaps in the $(h,r)$-plane for various values of $\sigma$.    We include the relevant parts of the Conjectures from \cite{AndWoo}.  

We recall some notation.  Let $[x]$ be the result of rounding $x$ to the nearest integer.  (This is a minor change from the notation in \cite{AndWoo}, where we originally used $(x)$.)  A {\em persistently missing point} $(h_0, r_0)$ in the $(h,r)$-plane is a point satisfying $(h_0, r_0)\not\in \mathcal{K}_\sigma$ for all $\sigma\ge \sigma_0$ for some constant $\sigma_0$.  

\medskip
\noindent
{\bf \cite[Conjecture 3.12]{AndWoo}} {\em For $\sigma\ge 9$, let $E_\sigma$ be the line with slope $-3$ passing through $(1, \sigma -1)$ and let $D_\sigma$ be the line with slope $-4$ passing through $(1, \sigma -1)$.  Then no point strictly between $E_\sigma$ and $D_\sigma$ lies in $\mathcal{K}_\sigma$.}

\begin{proof}
Note that $E_\sigma =L^\sigma_3$ and $D_\sigma =U^\sigma_4$ in our current terminology.  Conjecture 3.12  follows immediately from Theorem \ref{main gap theorem}, since the points strictly between $E_\sigma =L^\sigma_3$ and $D_\sigma =U^\sigma_4$ form the gap $\mathcal{G}_\sigma (3,4)$. 
\end{proof}

\medskip
\noindent
{\bf \cite[Conjecture 3.13]{AndWoo}} {\em  The point $(2, [\frac{2}{3}\sigma -4])$ is persistently missing for all $\sigma\ge 7$.}

\medskip
\noindent
{\bf \cite[Conjecture 3.14]{AndWoo}} {\em The points $(3, [\frac{2}{3}\sigma -7])$ and $(3, [\frac{2}{3}\sigma -8])$ persistently missing for all $\sigma\ge 18$.  For $\sigma\equiv 2 \: ({\rm mod}\: 3)$, the point $(3, [\frac{2}{3}\sigma -6])$ is persistently missing for all $\sigma\ge 18$.}

\begin{proof} For both Conjecture 3.13 and Conjecture 3.14, we consider the gap $\mathcal{G}_\sigma (4,6)$, which consists of the points strictly between the lines $L^\sigma_4$ and $U^\sigma_6$, with the exception of those points that lie on the line $L_\sigma (5,1)$.  The equation for $L_4^\sigma$ is $3r + 8h = 2\sigma +6$ and the equation for $U_6^\sigma$ is $3r + 12h = 2\sigma +10$.  

Setting $h=2$ (which is the case of interest in Conjecture 3.13) and solving for $r$,  we see that $(2,r)$ lies in the gap $\mathcal{G}_\sigma (4,6)$ for all 
\[ \frac{2}{3}\sigma - \frac{14}{3} < r < \frac{2}{3}\sigma -\frac{10}{3}.\]
Since 
\[ \frac{2}{3}\sigma - \frac{14}{3} < \left[\frac{2}{3}\sigma -4\right] = \left[\frac{2}{3}\sigma -\frac{12}{3}\right] < \frac{2}{3}\sigma -\frac{10}{3}\]
for all $\sigma\ge 7$, we see that $(2, [\frac{2}{3}\sigma -4])$ lies in $\mathcal{G}_\sigma (4,6)$ for all $\sigma \ge 7$.

Setting $h=3$ (which is the case of interest in Conjecture 3.14) and solving for $r$, we see that $(3,r)$ lies in the gap $\mathcal{G}_\sigma (4,6)$ for all 
\[ \frac{2}{3}\sigma -\frac{26}{3} < r < \frac{2}{3}\sigma -6.\]
Since 
\[ \frac{2}{3}\sigma -\frac{26}{3} < \left[\frac{2}{3}\sigma -8\right] < \left[\frac{2}{3}\sigma -7\right] < \frac{2}{3}\sigma -6 \]
for all $\sigma\ge 18$, we see that $(3,[\frac{2}{3}\sigma -8])$ and $(3, [\frac{2}{3}\sigma -7])$ both lie in $\mathcal{G}_\sigma (4,6)$ for all $\sigma\ge 18$.

This leaves the case in which  $\sigma\equiv 2 \: ({\rm mod}\: 3)$, where we see that 
\[ \frac{2}{3}\sigma -\frac{26}{3} < \left[\frac{2}{3}\sigma -6\right] <  \frac{2}{3}\sigma -6,\]
so that $(3,[\frac{2}{3}\sigma -6])$ lies in $\mathcal{G}_\sigma (4,6)$ for all $\sigma\ge 18$ satisfying $\sigma \equiv 2 \: ({\rm mod}\: 3)$.
\end{proof}

We note that this approach to the description of gaps extends naturally to higher values of $h$ and provides an essentially complete description of the significant gaps that we have seen from the available data.  

\section{The Line $r=1$}
\label{sporadic r=1}

The purpose of this Section is to extend Theorem 3.8 from \cite{AndWoo}, in which we show that the skeletal signature $(1,1)$ is sporadic, to all points of the form $(h,1)$ for $h\ge 1$.  This resolves Conjecture 3.15 from \cite{AndWoo}.  By {\em sporadic}, we mean a skeletal signature $(h_0, r_0)$ for which there are infinitely many genera $\sigma$ with $(h_0,r_0)\in \mathcal{K}_\sigma$ and infinitely many genera $\sigma$ with $(h_0,r_0)\not\in \mathcal{K}_\sigma$.

Let $X$ be a closed Riemann surface of genus $\sigma\ge 2$ and let $G$ be a finite group acting on $X$ by conformal homeomorphisms where the action of $G$ on $X$ has signature $(h; n_1,\ldots, n_r)$. A vector $(a_{1}, b_{1}, a_{2}, b_{2},\dots , a_h, b_h, c_{1},\dots , c_{r})$ of elements of $G$  is an {\em $(h;n_{1},\dots ,n_{r})$-generating vector for $G$} if the following hold:

\begin{enumerate}
\item $G=\langle a_{1},b_{1},a_{2},b_{2},\dots ,a_{h},b_{h},c_{1},\dots ,c_{r} \rangle$;
\item The order of $c_{i}$ is $n_{i}$ for $1\leq i\leq r$;
\item $\prod_{i=1}^{h} [a_{i} ,b_{i} ] \prod_{j=1}^{r} c_{j}$=1.
\end{enumerate}

In \cite{Bro2}, an adapted version of Riemann's existence theorem provides necessary and sufficient conditions for existence of group actions in terms of generating vectors. Specifically, there exists an action of $G$ on $X$ with signature $(h; n_1,\ldots, n_r)$ if and only if both the Riemann-Hurwitz formula is satisfied and there exists an $(h;n_{1},\dots ,n_{r})$-generating vector for $G$.

We are now ready to prove the main result of this Section.

\begin{theorem} For $\sigma\ge 2$, every point $(h,1)$ for $h >1$ is a sporadic point in $\mathcal{K}_\sigma$.
\label{sporadic points}
\end{theorem}

\begin{proof}
We shall first show that there exists an infinite sequence of genera for which there does not exist a group action with skeletal signature $(h,1)$ on a closed Riemann surface of any genus in this sequence. 

Suppose then that a group $G$ acts with skeletal signature $(h,1)$ on some Riemann surface of genus $\sigma$, so that the action of $G$ on $X$ has signature $(h;n)$ for some $n> 1$ and the Riemann-Hurwitz formula (\ref{riemann hurwitz}) is satisfied.  Solving, we see that
\[  \frac{2n(\sigma -1)}{n(2h-1)-1}=|G|.\]
Since $n(2h-1)-1$ and $n$ are relatively prime, it follows that $n(2h-1)-1$ must divide $2(\sigma -1)$.

Now, set $\sigma=p+1$ where $p$ is an odd prime. Since $n(2h-1)-1$ must divide $2(\sigma -1) =2p$, it follows that one of following four cases must hold, so that either $n(2h-1)-1=1$, $n(2h-1)-1=2$, $n(2h-1)-1=p$ or $n(2h-1)-1=2p$. 

In the first case, we see that  $n=2$ and $h=1$, while in the second case we see that $n=3$ and $h=1$.  In particular, we have that  $h=1$ for both cases, contrary to our initial assumption that $h >1$. 

For the remaining two cases we modify the proof given in \cite{AndWoo}  to prove no such action can exist. First note that if $G$ acts on $X$ with signature $(h;n)$ for some $n\ge 2$, then  there exists an $(h;n)$-generating vector for $G$ as defined above, for which $c_{1}$ is a commutator of $G$ of order $n$ (since $\left( \prod_{i=1}^{h} a_{i}b_{i}a_{i}^{-1}b_{i}^{-1}\right) c_{1} =e_G$). Since $n\geq 2$, it follows that $G$ cannot be abelian. 

For $n(2h-1)-1=p$, we have \[ |G|  =  \frac{2n p}{n (2h-1)-1} =  \frac{2p\left( \frac{p+1}{2h-1} \right)}{\left( \frac{p+1}{2h-1} \right)(2h-1)-1}= 2\left( \frac{p+1}{2h-1} \right)=2n.\] Hence, we see that $G$ has order $2n$ and acts on $X$ with signature $(1;n)$, from which it follows that $G$ has an index $2$ cyclic subgroup $H$ which contains the commutator $c_1$. However, since $n$ is even, we can show this is impossible imitating identically the proof given in \cite{AndWoo}. Hence $(h,1)$ is not a skeletal signature for $G$.

Finally, for the last case, we have $$n=\frac{2p+1}{2h-1}$$ and consequently 
\[ |G|= \frac{2n p}{n_i(2h-1)-1} = \frac{2p\left( \frac{2p+1}{2h-1} \right)}{\left( \frac{2p+1}{2h-1} \right)(2h-1)-1}=\frac{2p\left( \frac{2p+1}{2h-1} \right)}{2p}=\frac{2p+1}{2h-1}=n.\]
Since $G$ contains an element equal to its order, it must be cyclic and hence abelian, which is impossible.  

We have now shown that $(h,1)$ is not the skeletal signature of any group action for any genus of the form $\sigma=p+1$ for any odd prime $p$, of which there are infinitely many.

To complete the proof, we need to construct an infinite sequence of genera for which $(h,1)$ is a skeletal signature. Again we can generalize the proof for the case $(1,1)$ given in \cite{AndWoo}.  For $n\ge 2$, let 
\[ G_{n}=\langle x,y\: |\: x^n=y^2,y^{-1}xy=x^{-1} \rangle\]
denote the generalized quaternion group. Taking $a_{i}=b_{i}=e_{G}$ for $i>1$, the vector 
\[ (x,y,e_{G},e_{G},\dots ,e_{G},yx^{-2}y^{-1})\]
is then  an $(h;n)$-generating vector for $G_{n}$ (since $(x,y,yx^{-2}y^{-1})$ is a $(1,n)$-generating vector for $G_{n}$, as was proved in \cite{AndWoo}). Applying the Riemann-Hurwitz formula (\ref{riemann hurwitz}), it follows that $G_{n}$ acts on a surface of genus $\sigma =2n(2(h-1)+1)-1$. In particular, $(h,1)$ is a skeletal signature for $\sigma =2n(2(h-1)+1)-1$ for all $n\geq 2$.
\end{proof}

\begin{figure}
\begin{center}
\centerline{\epsfig{file =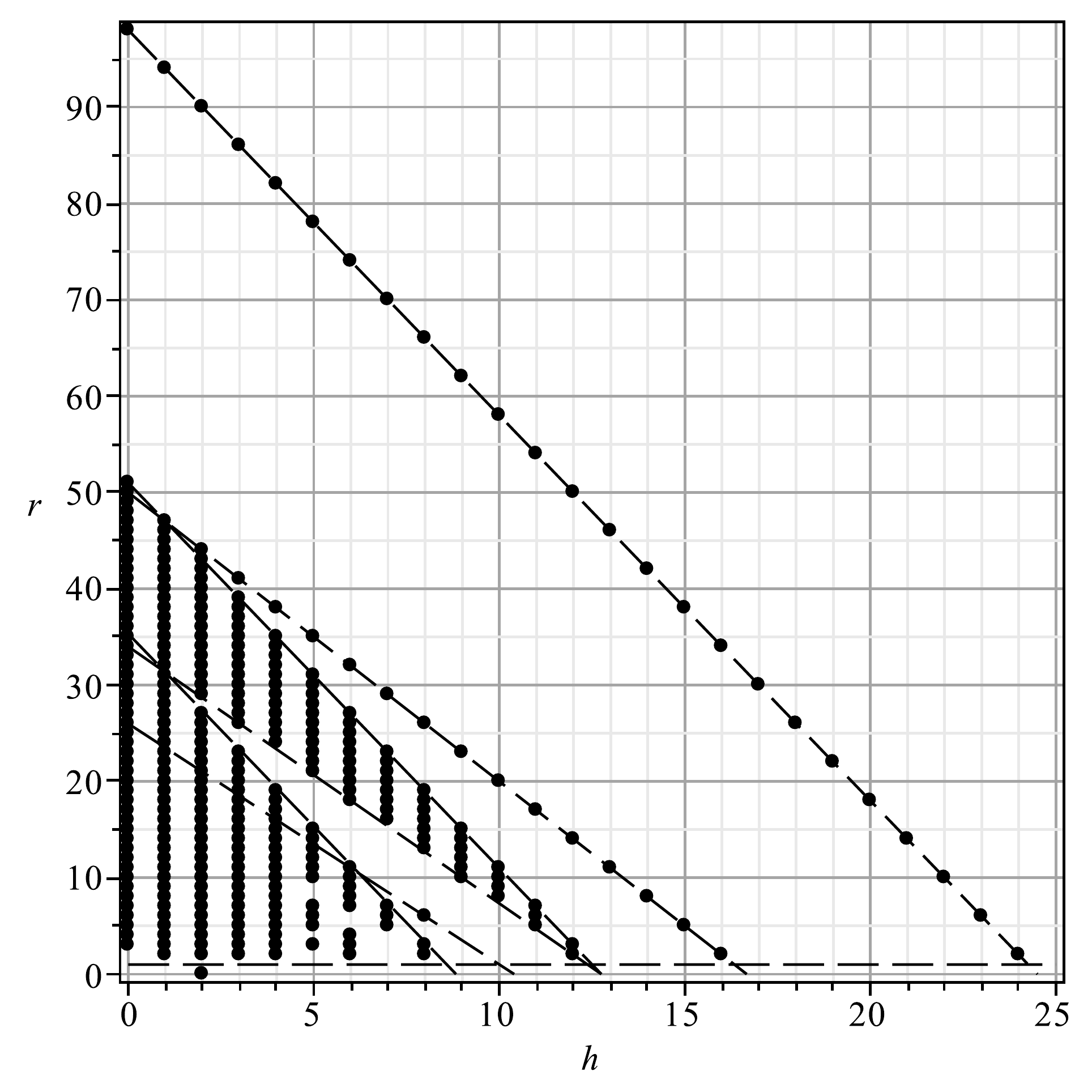, width=12cm, angle= 0}}
\end{center}
\caption{The $(h,r)$-plane for genus $\sigma = 48$.  The top line is the {\em hyperelliptic line} $L_{48}(2,1)$, containing the skeletal signatures for all $C_2$ actions on $X$.   The next two lines are the lines $L_3^{48}$ and below it $U_4^{48}$, bounding the gap $\mathcal{G}_{48}(3,4)$.   Below $U^{48}_4$ is $L^{48}_4$, which together with $U^{48}_4$ bounds the triangular region $\mathcal{P}_{48}(4)$, which here is completely filled by skeletal signatures.  Next is the line $U^{48}_6$, which together with $L^{48}_4$ bounds the gap $\mathcal{G}_{48}(4,6)$, into which intrudes the line $L^{48}(5) = L_{48}(5,1)$ bringing into $\mathcal{G}_{48}(4,6)$ the single skeletal signature $(8,6)$.  We also include the line $r=1$, to illustrate that there are no skeletal signatures of the form $(h,1)$ in genus $48$, as per the proof of Theorem \ref{sporadic points}.  This Figure was produced using MAPLE 15 and the data from the {\bf genus} package in GAP \cite{GAP}.}  
\label{fig:example}
\end{figure}


\begin{thebibliography}{99}

\bibitem{AndWoo} J W Anderson and A Wootton, \textit{A Lower Bound for the Number of Group Actions on a Compact Riemann Surface}, {\em Algebr. Geom. Topol.} {\bf 12} (2012) 19--35.

\bibitem{AndWooskeletal} J W Anderson and A Wootton, in preparation.

\bibitem{Breu} T Breuer. \textit{Characters and Automorphism Groups of Compact Riemann Surfaces}, London Mathematical Society Lecture Note Series {\bf 280}, Cambridge University Press (2001).

\bibitem {Bro2} S A Broughton, \textit{Classifying Finite Group Actions on Surfaces of Low Genus}, {\em J. Pure  Appl. Algebra} {\bf 69} (1990), 233--270.

\bibitem{GAP} The GAP Group. \textit{GAP - Groups, Algorithms and Programming}, Version 4.4, 2006 (\textsf{http://www.gap-system.org})

\bibitem{Gur} R Guralnick, \textit{On the number of generators of a finite group} {\em Arch. Math.} {\bf 53} (1989), 521--523

\bibitem{Har1} W J Harvey, \textit{Cyclic groups of automorphisms of a compact Riemann surface}
{\em Quart. J. Math. Oxford } {\bf 17} (1966), 86--97

\end{thebibliography}
\end{document}